\documentclass[12pt]{amsart}
\usepackage{amssymb}
\usepackage{longtable}
\newcommand{\g}{\mathfrak{g}}
\newcommand{\Orb}{\mathbb{O}}
\newcommand{\K}{\mathbb{K}}
\newcommand{\Walg}{\mathcal{W}}
\newcommand{\Str}{\mathcal{O}}
\newcommand{\Dcal}{\mathcal{D}}
\newcommand{\Der}{\operatorname{Der}}
\newcommand{\U}{\mathcal{U}}
\newcommand{\gr}{\operatorname{gr}}
\newcommand{\Spec}{\operatorname{Spec}}
\newcommand{\Quant}{{\mathcal{Q}}}

\newcommand{\A}{\mathcal{A}}
\newcommand{\F}{\operatorname{F}}
\renewcommand{\sl}{\mathfrak{sl}}
\newcommand\ad{\operatorname{ad}}
\newcommand\Z{\mathbb{Z}}
\newcommand\SL{\operatorname{SL}}
\newcommand\Aut{\operatorname{Aut}}
\newcommand\z{\mathfrak{z}}
\newcommand\q{\mathfrak{q}}
\newcommand\m{\mathfrak{m}}
\newcommand\im{\operatorname{im}}
\newcommand\W{\mathbf{A}}
\newcommand\VA{\operatorname{V}}

\newcommand\Id{\mathfrak{Id}}
\newcommand\M{\mathcal{M}}

\newcommand\Jet{\operatorname{J}}
\newcommand\Dsh{\mathfrak{D}}
\newcommand\Ush{\mathfrak{U}}
\newcommand\J{\mathcal{J}}
\newcommand\I{\mathcal{I}}
\newcommand\Ish{\mathfrak{I}}
\newcommand\Hom{\operatorname{Hom}}
\newcommand\lf{\mathfrak{l}}
\newcommand\h{\mathfrak{h}}
\newcommand\Ann{\operatorname{Ann}}
\renewcommand\t{\mathfrak{t}}
\newcommand\KF{\operatorname{F}}
\newtheorem{Thm}{Theorem}[section]
\newtheorem{Prop}[Thm]{Proposition}
\newtheorem{Cor}[Thm]{Corollary}
\newtheorem{Lem}[Thm]{Lemma}
\theoremstyle{definition}

\newtheorem{Rem}[Thm]{Remark}

\numberwithin{equation}{section}
\title{Quantizations of nilpotent orbits vs 1-dimensional representations of W-algebras}
\author{Ivan Losev}
\thanks{Supported by the NSF grant DMS-0900907}
\thanks{Address: MIT, Dept. of Math., 77 Massachusetts Avenue, Cambridge MA02139, USA}
\thanks{E-mail: ivanlosev@math.mit.edu}
\begin{document}
\begin{abstract}
Let $\g$ be a semisimple Lie algebra over an algebraically closed field $\K$
of characteristic 0 and $\Orb$ be a nilpotent orbit in $\g$. Then $\Orb$ is a symplectic
algebraic variety and one can ask whether it is possible to quantize $\Orb$ (in an appropriate sense)
and, if so, how to classify the quantizations. On the other hand, for the pair $(\g,\Orb)$ one can construct an associative algebra $\Walg$ called a (finite) W-algebra. The goal of this paper is to clarify a relationship between quantizations of $\Orb$ (and of its coverings)  and 1-dimensional $\Walg$-modules. In the first approximation, our result  is that there is a one-to-one correspondence between the two.  The result is not new: it was discovered
(in a different form) by Moeglin in the 80's.
\end{abstract}
\maketitle
\markright{QUANTIZATIONS OF NILPOTENT ORBITS VS W-ALGEBRAS}
\section{Introduction}
We fix a base field $\K$ which is assumed to be algebraically closed and of characteristic 0.
Let $\g$ be a semisimple Lie algebra, $G$  the corresponding simply connected algebraic group,
and $\Orb\subset\g\cong \g^*$  a nilpotent orbit. The variety $\Orb$ is symplectic with respect to the Kostant-Kirillov form.  So one can pose the problem of quantizing $\Orb$.

The purpose of this paper is to relate two types of objects:
\begin{itemize}
\item Some special quantizations of $\Orb$ and, more generally, of its $G$-equivariant  coverings,
%\item certain Dixmier algebras,
\item 1-dimensional representations of a so called {\it W-algebra} $\Walg$ constructed
from $\g$ and $\Orb$.
\end{itemize}

There are several (related but different) notions of quantizations.
In this paper we are mostly going to deal with Deformation quantization in the algebro-geometric
setting. So our quantization $\Dcal$ of $\Orb$ will be a formal deformation of the
structure sheaf $\Str_{\Orb}$ of Poisson algebras. We will require quantizations
to be compatible with $G$ and $\K^\times$-actions on $\Orb$, the $\K^\times$-action on $\Orb$ being induced from the usual action on $\g$ by multiplications by scalars. All necessary definitions will be given in Subsection \ref{SUBSECTION_quant_generalities}.
Apart from nilpotent orbits themselves we will also consider their $G$-equivariant coverings.
For a $G$-equivariant covering $X$ of $\Orb$ we denote the set of isomorphism classes of
quantizations (satisfying the conditions mentioned above)  of $\Orb$ by $\Quant(X)$.

Deformation quantization has been studied extensively  starting from the seminal
paper \cite{BFFLS} mostly in the $C^\infty$-setting. Some work was done in the algebraic setting,
see, e.g., \cite{Kontsevich}, \cite{BK}, \cite{Ye}. In particular, in \cite{BK} it was shown
that a symplectic algebraic variety $X$ can be quantized provided it satisfies some vanishing-like
condition on the cohomology groups $H^i(X,\Str_X)$. Unfortunately, nilpotent orbits and their
coverings are very far from satisfying these conditions.

%The second object we are going to deal with are some special {\it Dixmier} algebras,
%the definition is recalled in Subsection \ref{SUBSECTION_Dixmier}. In particular,
%we can equip a Dixmier algebra with a good algebra filtration. We say that a pair of a Dixmier
%algebra $\A$ and its  good algebra filtration {\it quantizes} a covering $\widetilde{\Orb}$ of $\Orb$ if $\gr\A$ is a %domain and an open $G$-orbit in   $\Spec(\gr \A)$ (unique, thanks to the first condition)
%is isomorphic to $\widetilde{\Orb}$. One can introduce a partial ordering
%on the set of (isomorphism classes of) quantizing pairs (roughly speaking,
%by inclusion). We are interested in {\it maximal} pairs. The set of their isomorphism
%classes will be denoted by $\MQD(\widetilde{\Orb})$.

The second object we are interested in is 1-dimensional modules over a W-algebra
$\Walg$. Here we  give a very brief reminder on  W-algebras,  more details  will be given in Subsection \ref{SUBSECTION_W_generalities}.

Choose an element $e\in \Orb$ and an $\sl_2$-triple $(e,h,f)$. Recall the Slodowy slice
$S:=e+\ker \ad(f)$. The algebra $\K[S]$ of regular functions on $S$ has a natural grading
called the {\it Kazhdan grading}. A W-algebra is a particularly nice filtered algebra $\Walg$
with $\gr \Walg=\K[S]$. One of its features which will be important for us is that the group
$Q:=Z_G(e,h,f)$ acts on $\Walg$ by automorphisms.
Let $\Id^1(\Walg)$ denote the set of two-sided ideals of codimension 1 in $\Walg$, this set is naturally
identified with the set of isomorphism classes of 1-dimensional $\Walg$-modules). The group $Q$ acts naturally
$\Id^1(\Walg)$. It turns out that this action descends to the component group
$C(e):=Q/Q^\circ$.

The  main result of this paper is the following theorem.

\begin{Thm}\label{Thm:main1}
Let $\widetilde{\Orb}$ be a $G$-equivariant covering of $\Orb$. Pick a point $x\in \widetilde{\Orb}$
lying over $e$ and set $\Gamma=G_x/(G_x)^\circ$. Then
$\Quant(\widetilde{\Orb})$ is in bijection with the set $\Id^1(\Walg)^\Gamma$ of
$\Gamma$-fixed points in  $ \Id^1(\Walg)$.
\end{Thm}

One-dimensional $\Walg$-modules have been studied extensively during the last few years see
\cite{GRU}, \cite{Wquant},\cite{Miura}, \cite{Premet2}, \cite{Premet4}. We will recall these results
in Subsection \ref{SUBSECTION_W_onedim}.

It turns out that Theorem \ref{Thm:main1} is not essentially new: a closely related result was obtained by
Moeglin in \cite{Moeglin2}. She used a different language to speak about quantizations, namely,
the language of {\it Dixmier algebras}. Also she dealt with so called {\it Whittaker modeles}
of primitive ideals in $U(\g)$ rather than with 1-dimensional $\Walg$-modules. We will
explain a relation of Moeglin's work to ours in Subsection \ref{SUBSECTION_Moeglin}.
Our proof is different from Moeglin's and (at least, from the author point of view)
easier and more natural. It is based on the construction of the  {\it jet bundle} of a quantization
and also on the decomposition theorem from \cite{Wquant}. The latter is a very basic result
about the relationship between the universal enveloping algebras and W-algebras. %Second,
%now we know much more about 1-dimensional $\Walg$-modules than 20 years ago.

This paper is organized as follows. We list the notation
used in the paper in Section \ref{SECTION_Notation}. In Section \ref{SECTION_Def_quant}
we provide some preliminary results on Deformation quantization. In Section \ref{SECTION_W}
we recall some known results and constructions related to W-algebras. Finally,
in Section \ref{SECTION_proofs} we prove the main theorem as well as some other related results.
In the beginning of each section its content is described in more detail.

{\bf Acknowledgements.} I would like to thank V. Dolgushev, P. Etingof and D. Vogan for useful discussions.

\section{Notation}\label{SECTION_Notation}
%Let $X$ be an algebraic variety. By a pro-coherent sheaf on $X$ we mean an inverse limit of coherent
%sheaves.
\begin{longtable}{p{3cm} p{12cm}}
$\widehat{\otimes}$&the completed tensor product of complete topological vector spaces/ modules.\\
%$(a_1,\ldots,a_k)$& the two-sided ideal in an associative algebra generated by  elements $a_1,\ldots,a_n$.\\
% $A^\wedge_\chi$&
%the completion of a commutative algebra $A$ with respect to the maximal ideal
%of a point $\chi\in \Spec(A)$.\\
%$\Ann_\A(\M)$& the annihilator of an $\A$-module $\M$ in an algebra
%$\A$.\\
%$\Cl(X)$& the "classical part" of a subset $X$ in a
%$\K[\hbar]$-module $\A$, that
%is, the image of $X$ in $\A/\hbar\A$.\\
$\Aut(Y)$& the automorphism group of an object $Y$.\\
$\Der(A)$& the Lie algebra of derivations of an algebra $A$.\\
%$G*_HV$&$:=(G\times V)/H$: the homogeneous vector bundle over $G/H$ with fiber $V$.\\
%$g*_Hv$&the class of $(g,v)\in G\times V$ in $G*_HV$.\\
$G_x$& the stabilizer of $x$ in $G$.\\
%$\Goldie(\A)$& the Goldie rank of a prime Noetherian algebra $\A$.\\
$\gr \A$& the associated graded vector space of a filtered
vector space $\A$.\\
%$I(Y)$& the ideal in $\K[X]$ consisting of all functions vanishing
%on $Y$ for a subvariety $Y$ in an affine variety $X$.\\
$H^i_{DR}(X)$& the $i$-th De Rham cohomology of a smooth algebraic variety $X$.\\
$\Id(\A)$& the set of all (two-sided) ideals of an algebra $\A$.\\
$\M_{\g-fin}$& the  finite part of a
$\g$-module $\M$.\\
$N_G(H)$& the normalizer of a subgroup $H$ in a group $G$.\\
$\Str_X$& the structure sheaf of an algebraic variety $X$.
\\$R_\hbar(\A)$&$:=\bigoplus_{i\in
\mathbb{Z}}\hbar^i \F_i\A$ :the Rees vector space of a filtered
vector space $\A$.
\\$U(\g)$& the universal enveloping algebra of a Lie algebra $\g$.
\\$\VA(\M)$& the associated variety of a $U(\g)$-module $\M$.\\
$Z_G(H)$& the centralizer of a subgroup $H$ in a group $G$.\\
%$\Centr(\g)$& the center of $U(\g)$.\\
$\Gamma(X,\mathcal{F})$& the space of global sections of a sheaf
$\mathcal{F}$ on $X$.
\end{longtable}

\section{Deformation quantization}\label{SECTION_Def_quant}
The material of this section is basically standard.

In Subsection \ref{SUBSECTION_quant_generalities} we recall the notion of
a deformation quantization of a symplectic algebraic variety. Also we will define
 {\it $G$-equivariant} and of a {\it graded} quantizations
and also   quantum comoment maps.  We finish the subsection recalling some
facts about the Fedosov quantization of affine varieties.

In Subsection \ref{SUBSECTION_qjb} we will recall the notion of jet bundle of a quantization that we
learned from \cite{BK}. Then axiomatizing its properties we introduce
the notion of a quantum jet bundle, which plays a crucial in the proof of Theorem \ref{Thm:main1}.

Finally, in Subsection \ref{SUBSECTION_qcm} we will prove a result on the existence
and the uniqueness of a quantum comoment map which are used in the proof of Theorem \ref{Thm:main1}.

\subsection{Generalities}\label{SUBSECTION_quant_generalities}
Let $X$ be a smooth algebraic variety equipped with a symplectic form $\omega$.
Let $\{\cdot,\cdot\}_\omega$ denote the Poisson bracket on the structure sheaf $\Str_X$ induced by
$\omega$.

 Suppose we are given
a  sheaf $\Dcal$ of $\K[[\hbar]]$-algebras on $X$ together with an isomorphism $\Dcal/\hbar \Dcal\rightarrow \Str_X$.
We suppose that $\Dcal$ is flat over $\K[[\hbar]]$ and is complete
 and separated in the $\hbar$-adic topology and such that
an isomorphism $\theta: \Dcal/ \hbar\Dcal\rightarrow \Str_{X}$ is fixed. Further, we assume that $[\widetilde{a},\widetilde{b}]$
is divisible by $\hbar^2$ for any local sections $\widetilde{a},\widetilde{b}$ of $\Dcal$.
Then the identification $\Str_X\cong \Dcal/\hbar \Dcal$ gives rise to a
Poisson bracket on $\Dcal$. Namely, fix $x\in X$ and pick two local sections $a,b$ of $\Str_X$ on a neighborhood
$U$ of $x$. Shrinking $U$ if necessary, we may assume that $a,b$ can be lifted to
sections $\widetilde{a},\widetilde{b}$ of $\Dcal$. Then we set $\{a,b\}_{\Dcal}:=\frac{1}{\hbar^2}[\widetilde{a},\widetilde{b}] \mod \hbar$.
We say that $\Dcal$ (or, more precisely, the pair $(\Dcal,\theta)$) is a {\it quantization} of the symplectic variety $X$ if $\{\cdot,\cdot\}_\omega=\{\cdot,\cdot\}_{\Dcal}$.

Usually the definition of a  quantization is given in a slightly different way, one requires
$[\widetilde{a},\widetilde{b}]$ to be divisible by $\hbar$. Given the quantization in that sense
we can get a quantization in our sense by replacing $\hbar$ with $\sqrt{\hbar}$. So our
notion of quantization is (slightly) more restrictive.  We choose our convention because it makes
passing from filtered $\K$-algebras to graded $\K[\hbar]$-algebras by using the Rees construction
more convenient. Peculiarly, we need to use the same convention while working with W-algebras, see \cite{Wquant},
\cite{HC}.

We assume that $X$ comes equipped with two group actions. First, let an algebraic group
$G$ act on $X$ by symplectomorphisms. Second, suppose that the one-dimensional torus $\K^\times$
acts on $X$ commuting with the $G$-action such that $t.\omega=t^2\omega$ for all $t\in \K^\times$.

In this paper we are going to consider graded Hamiltonian $G$-equivariant quantizations.
Let us explain what we mean by this.

We say that a quantization $\Dcal$ of $X$ is {\it $G$-equivariant} if the action of $G$ on $\Str_{X}$
lifts to a $G$-action on $\Dcal$ by algebra automorphisms such that $\hbar$ is $G$-invariant.
In particular, the $G$-action on $\Dcal$ gives rise to a Lie algebra
homomorphism $\g\rightarrow \Der(\Dcal)$, the image of $\xi\in\g$ under this homomorphism by $\xi_{\Dcal}$.

Now let us explain what we mean by a {\it Hamiltonian} quantization.
Recall that a morphism $\mu:X\rightarrow \g^*$ is called a
{\it moment map} if $\mu$ is $G$-equivariant and has the following property: for $\xi\in \g$ and  a local
section $f$ of $\Str_X$ the equality $\{\mu^*(\xi), \cdot\}=\xi_{X}$ holds (here $\xi_X$ is the derivation
of $\Str_X$ induced by the $G$-action).

Suppose that $X$ is equipped with a moment map $\mu:X\rightarrow \g^*$.
A $G$-equivariant quantization $\Dcal$ of $X$ is called
{\it Hamiltonian} if it is equipped with a $G$-equivariant homomorphism $\varphi_\hbar:\g\rightarrow \Gamma(X,\Dcal)$
(a quantum comoment map) satisfying $\frac{1}{\hbar^2}[\varphi(\xi),\cdot]=\xi_{\Dcal}$
for all $\xi\in\g$ and coinciding with $\mu^*$ modulo $\hbar$.

In the next section we will see that under some conditions on $X$ and on $G$ any $G$-equivariant
quantization $\Dcal$ of $X$ a quantum comoment map exists and is unique.

Now let us define {\it graded} Hamiltonian $G$-equivariant quantizations.

We say that a $G$-equivariant Hamiltonian quantization $\Dcal$ of $X$  is
graded if $\Dcal$ is equipped with a $\K^\times$-action by algebra automorphisms  satisfying
the following conditions.
\begin{enumerate}
\item $t.\hbar=t\hbar$.
\item The $\K^\times$-action  commutes with the $G$-action.
\item $t.\varphi_\hbar(\xi)=t^2 \varphi_\hbar(\xi)$ for all $t\in \K^\times,\xi\in\g$.
\item The isomorphism $\theta:\Dcal/\hbar\Dcal\xrightarrow{\sim} \Str_{X}$ is $\K^\times$-equivariant.
\end{enumerate}

Two graded Hamiltonian $G$-equivariant quantizations $(\Dcal_1,\theta_1),(\Dcal_2,\theta_2)$
 with quantum comoment maps $\varphi^1_\hbar,\varphi^2_\hbar$ of $X$ are said to be isomorphic if there is a $G\times \K^\times$-equivariant
isomorphism $\psi:\Dcal_1\rightarrow \Dcal_2$ of sheaves of $\K[[\hbar]]$-algebras such that
$\theta_1=\theta_2\circ\psi, \varphi^1_\hbar=\varphi^2_\hbar\circ \psi$.
All graded Hamiltonian $G$-equivariant  quantizations of $X$ form a category (in fact, a groupoid).

Until a further notice $G$ denotes a simply connected semisimple algebraic group.
We identify $\g$ with $\g^*$ using the Killing form.
Let us provide a class of examples of algebraic varieties equipped with the structures described above.
A variety $X$ we are going to consider will be either $\Orb$ or, more generally, some
$G$-equivariant covering of $\Orb$.

Let  $\eta:X\twoheadrightarrow \Orb$
be the projection. Then $\omega:=\eta^*\omega_{KK}$, where $\omega_{KK}$ stands for the Kostant-Kirillov
form on $\Orb$, is a symplectic form on $X$. A natural $G$-action on $X$ preserves $\omega$.
The composition of $\eta$ with the inclusion $\Orb\hookrightarrow \g^*$ is a moment map.

Next, let us introduce a $\K^\times$-action on $X$ by $G$-equivariant automorphisms.
In other words, we need a homomorphism from $\K^\times$ to the group $\Aut^G(X)$ of $G$-equivariant
automorphisms of $X$. The latter is naturally identified with
$N_G(H)/H$, where $H$ is the stabilizer of some point $x\in X$.

Set $e:=\eta(x)$. This is a nilpotent element in $\g$. Fix an $\sl_2$-triple $(e,h,f)$.
Let $\gamma:\K^\times\rightarrow G$ denote the composition
of the homomorphism $\SL_2\rightarrow G$ induced by the $\sl_2$-triple and of the embedding
$\K^\times\hookrightarrow \SL_2, t\mapsto \begin{pmatrix} t&0\\0&t^{-1}
\end{pmatrix}$. In particular, $\gamma(t).\xi= t^i\xi$ provided $[h,\xi]=i\xi$ for $\xi\in\g, i\in \Z$.
We remark that the image of $\gamma$ lies in $N_G(Z_G(e))$. Moreover, $\gamma(t)$ normalizes $H$.
Indeed, $\gamma(t)$ normalizes $Z_G(e)$ and hence $Z_G(e)^\circ$. This gives rise to a homomorphism
$\K^\times\rightarrow \Aut(C(e))$. But $\K^\times$ is connected and so the last homomorphism
is trivial. It follows that $\gamma(t)$ normalizes every subgroup between $Z_G(e)^\circ$ and
$Z_G(e)$, in particular, $H$.

So we can consider the composition $\underline{\gamma}: \K^\times\rightarrow N_G(H)\twoheadrightarrow
N_G(H)/H$. Define the action of $\K^\times$ on $X$ by means of $\underline{\gamma}^{-1}$. The
moment map $\mu:X\rightarrow\g^*$  becomes $\K^\times$-equivariant, where the action of $\K^\times$
on  $\g^*$ is given by $(t,\alpha)\mapsto t^{-2}\alpha$.  We remark that the most obvious $\K^\times$-action
on $\g^*$, $(t,\alpha)\mapsto t^{-1}\alpha$, in general, cannot be lifted to $X$: one can take
$\g=\sl_2$ and the universal covering of the principal orbit to get a counter-example.

We finish the subsection recalling the Fedosov construction of quantizations. We are interested
in the situation when the symplectic variety $X$ in interest is affine. In this case a quantization
$\Dcal$ of $X$ is isomorphic to $\Str_X[[\hbar]]$ as a sheaf of $\K[[\hbar]]$-modules. Moreover,
the algebra structure on $\Dcal$ is uniquely recovered from the algebra structure on the space
$\Gamma(X,\Dcal)=\K[X][[\hbar]]$ of global sections of $\Dcal$. These claimed are proved analogously
to Remarks 1.6,1.7 in \cite{BK}. We note that their setting is a conventional one: they have $\hbar$
while we have $\hbar^2$. Still, the arguments extend easily to our setting. This remark also hold for the
results recalled below.

The product on $\K[X][[\hbar]]=\Gamma(X,\Dcal)$ is usually called a {\it star-product}. Fedosov,
\cite{Fedosov1},\cite{Fedosov2}, constructed a star-product in the $C^\infty$-setting starting from
a symplectic connection on the tangent bundle and a closed 2-form $\Omega=\sum_{i=0}^\infty \omega_i \hbar^i$
with $\omega_0=\omega$ (the {\it curvature form} of a star-product). The same construction works in
the algebraic setting provided a symplectic connection exists that is always the case for affine
varieties, see, for example, \cite{Wquant}. Also if a reductive group $G$ acts on $X$ by symplectomorphisms,
then one can choose a $G$-invariant symplectic connection. The Fedosov construction
implies that the star-product produced from a $G$-invariant symplectic connection and
a $G$-invariant curvature form is $G$-equivariant.

It turns out that any star-product on $X$ is equivalent to a Fedosov one (meaning that the corresponding
quantizations are equivalent), and two Fedosov star-products
constructed from curvature forms $\Omega_1,\Omega_2$ are equivalent if and only if $\Omega_1-\Omega_2$ is exact.
The first claim follows, for example, from results of Kaledin and Bezrukavnikov, \cite{BK}, the second
one was proved by Fedosov in \cite{Fedosov2}. In the $G$-equivariant setting (when a star-product
is $G$-equivariant/ the curvature forms are $G$-invariant) and equivalence can also be made
$G$-equivariant.

In Subsection \ref{SUBSECTION_qcm} we will need an existence criterium for a quauntum comoment map.
Such a criterium was obtained by Gutt and Rawnsley in \cite{GR2}, Theorem 6.2. Again, their proof  transfers
to the algebraic setting (and to our definition of a star-product) directly.

\begin{Prop}\label{Prop_affine_qcm}
Let $G$ be a reductive group acting on $X$ by symplectomorphisms. Construct a star-product
on $X$ starting from a $G$-invariant symplectic connection and a $G$-invariant curvature form
$\Omega$. Then the quantum comoment  map for the $G$-action exists if and only if
$\iota_{\xi_X}\Omega$ is exact for all $\xi\in\g$ ($\iota_\bullet$ stands for the contraction).
Moreover, if the forms are exact, then a $G$-equivariant linear map $\varphi_\hbar:\g\rightarrow \K[X][[\hbar]]$ is a quantum comoment map if and only if $d\varphi_\hbar(\xi)=\iota_{\xi_X}\Omega$.
\end{Prop}

%A definition of isomorphic $\K^\times$- and $G$-equivariant quantizations of $\widetilde{\Orb}$ with quantum
%comoment maps is given in a straightforward way. The set of isomorphism classes is denoted by
%$\Quant(\widetilde{\Orb})$.

\subsection{Quantum jet bundles}\label{SUBSECTION_qjb}
In the proof of Theorem \ref{Thm:main1} an important role is played by the jet bundles of quantizations,
compare with \cite{BK}. The material of this subsection should be pretty standard.

We start by recalling the jet bundle $\Jet^\infty \Str_X$ of a smooth variety $X$.
Let $\pi_1,\pi_2$ denote the projections $X\times X\rightarrow X$ to the first and
to the second factor.
Consider the completion $\Str^\wedge_{X\times X}$ of $\Str_{X\times X}=\pi_2^*(\Str_X)$
with respect to the ideal $I_\Delta$ of the diagonal $X\hookrightarrow X\times X$, i.e.,
$\Str^\wedge_{X\times X}:=\varprojlim_{k\rightarrow \infty} \Str_{X\times X}/I_\Delta^k$.
By definition, $\Jet^\infty\Str_X:=\pi_{1*}(\Str^\wedge_{X\times X})$. This is a pro-coherent
sheaf of $\Str_X$-algebras. The jet bundle  $\Jet^\infty \Str_X$ comes equipped with a flat connection $\nabla$ defined as follows. Pick a vector field $\xi$ on $X$. Define the connection $\nabla$ on $\pi_{1*}(\Str_{X\times X})$
by $\nabla_\xi(f\otimes g)=(\xi.f)\otimes g$. This connection can be uniquely extended
to a continuous (with respect to the $I_\Delta$-adic topology) connection on $\Jet^\infty \Str_X$,
which is a connection we need. The sheaf of flat sections of $\Jet^\infty \Str_X$ is naturally
identified with $\Str_X$ (or, more precisely, with $\pi_{1*}(\pi_2^{-1}(\Str_X))$, where
$\pi_2^{-1}$ denotes the sheaf-theoretic pull-back).
Any fiber of $\Jet^\infty \Str_X$ is (non-canonically) identified with the algebra of formal
power series in $\dim X$ variables.

Suppose now that $X$ is a symplectic variety. Then $\Jet^\infty \Str_X$ comes equipped with a
$\Str_X$-linear Poisson bracket (extended by continuouty from the natural $\pi_1^{-1}(\Str_X)$-linear
Poisson bracket on $\Str_{X\times X}$). The induced bracket on $\Str_X$ (considered as the space
of flat sections in $\Jet^\infty \Str_X$) coincides with the initial one.
Any fiber of $\Jet^\infty \Str_X$ is isomorphic (as a Poisson algebra)
to the algebra $A:=\K[[x_1,\ldots,x_n,y_1,\ldots,y_n]],n:=\frac{1}{2}\dim X$, where a Poisson bracket
is defined by $\{x_i,x_j\}=\{y_i,y_j\}=0, \{x_i,y_j\}:=\delta_{ij}$.

Now let us define the notion of a {\it quantum jet bundle}.
By definition, a quantum jet bundle on $X$ is a triple $(\Dsh,\widetilde{\nabla},\Theta)$ consisting of
\begin{itemize}
\item a pro-coherent sheaf $\Dsh$ of $\Str_X[[\hbar]]$-algebras such that $[\Dsh,\Dsh]\subset \hbar^2\Dsh$,
\item a flat $\K[[\hbar]]$-linear connection $\widetilde{\nabla}$ on $\Dsh$ that is compatible
with the algebra structure in the sense that $\widetilde{\nabla}_\xi$ is a derivation of
$\Dsh$ for any vector field $\xi$ on $X$,
\item and an isomorphism $\Theta:\Dsh/\hbar\Dsh\rightarrow \Jet^\infty\Str_X$ of flat sheaves of Poisson
algebras ("flat" means the $\Theta$ intertwines the connections; the bracket on $\Dsh/\hbar\Dsh$ is defined
as on $\Dcal/\hbar \Dcal$, see the previous subsection).
\end{itemize}
By a pro-coherent sheaf we mean an inverse limit of coherent sheaves.
%satisfying the following conditions
%\begin{enumerate}
%\item $[\Dsh,\Dsh]\subset \hbar^2 \Dsh$. This gives rise to a Poisson bracket on $\Jet^\infty \Str_X$.
%\item $\Theta$ is an isomorphism of Poiss
%\end{enumerate}
Any fiber of $\Dsh$ is a quantization of the Poisson algebra $A$. One can show that any quantization
of $A$ is isomorphic to the formal Weyl algebra, i.e., the space $A[[\hbar]]$
equipped with a Moyal-Weyl product $f*g=\mu(\exp{\frac{P\hbar^2}{2}}(f\otimes g))$, but we are
not going to use this fact. Here $\mu:A\otimes A\rightarrow A$
is the multiplication map and $P$ is the Poisson bivector on $A$.

An isomorphism of two quantum jet bundles $(\Dsh_1,\widetilde{\nabla}_1,\Theta_1), (\Dsh_2,\widetilde{\nabla}_2,\Theta_2)$ is an isomorphism $\Psi:\Dsh_1\rightarrow \Dsh_2$ of $\Str_X[[\hbar]]$-algebras intertwining the connections, and such that $\Theta_1=\Theta_2\circ \varphi$.

Now let $\Dcal$ be a quantization of $X$. Then one can define the jet bundle $\Jet^\infty\Dcal$
of $\Dcal$ (see \cite{BK}, Definition 1.4) similarly to $\Jet^\infty\Str_X$. Namely consider the bundle
$\Str_X\otimes \Dcal$ on $X\times X$. There is a natural projection (the quotient by $\hbar$) $\Str_X\otimes \Dcal\twoheadrightarrow
\Str_X\otimes \Str_X$. Let $\widetilde{I}_\Delta$ denote the inverse image of
$I_\Delta$ under this projection. Set $(\Str_X\otimes \Dcal)^\wedge:=\varprojlim_{k\rightarrow \infty}\Str_X\otimes \Dcal/\widetilde{I}_\Delta^k$ and $\Jet^\infty \Dcal:=\pi_{1*}((\Str_X\otimes\Dcal)^\wedge)$.
The sheaf $\Jet^\infty \Dcal$ is equipped with a flat connection
$\widetilde{\nabla}$ defined completely analogously to the connection $\nabla$ above. There is a natural
isomorphism $\Theta:\Jet^\infty \Dcal/\hbar \Jet^\infty \Dcal\rightarrow \Jet^\infty \Str_X$.
It is easy to see that $(\Jet^\infty \Dcal,\widetilde{\nabla},\Theta)$ is a quantum jet bundle in the sense of the above definition.
The sheaf of flat sections of $\J^\infty \Dcal$ is  $\pi_{1*}(\pi_2^{-1}(\Dcal))$.

Conversely, let $(\Dsh, \widetilde{\nabla},\Theta)$ be a quantum jet bundle. Consider the sheaf $\Dcal:=\Dsh^{\widetilde{\nabla}}$ of flat sections of $\Dsh$. The isomorphism
$\Theta:\Dsh/\hbar\Dsh\rightarrow \Jet^\infty \Str_X$ restricts  to an embedding
$\theta:\Dcal/\hbar \Dcal\rightarrow (\Jet^\infty \Str_X)^\nabla=\Str_X$. It is not difficult
to show that $\theta$ is also surjective. So $(\Dcal,\theta)$ is a quantization of $X$.

This discussion can be summarized in the following proposition.

\begin{Prop}\label{Prop:jet1}
The assignments
\begin{itemize}
\item $\Dcal\mapsto \Jet^\infty \Dcal$
\item $\Dsh\mapsto \Dsh^{\widetilde{\nabla}}$
\end{itemize}
define mutually quasi-inverse equivalences between the category of quantizations
and the category of quantum jet bundles on $X$.
\end{Prop}

We need a ramification of this proposition covering graded Hamiltonian $G$-equivariant quantizations.
So we need the notion of a graded Hamiltonian $G$-equivariant quantum jet bundle.

Let $G$ be an algebraic group.
Suppose $G,\K^\times$ act on $X$ as in the previous subsection.
Clearly, a $G$-action induces an action of $G$ on $\Jet^\infty \Str_X$ (induced by
the diagonal $G$-action on $X\times X$). This action preserves the
connection and the algebra structure. We say that a quantum jet bundle $(\Dsh,\widetilde{\nabla},\Theta)$ is {\it $G$-equivariant} if $\Dsh$ is equipped with a $G$-action such that $\hbar$ and $\widetilde{\nabla}$ are $G$-invariant and $\Theta:\Dsh/\hbar \Dsh\rightarrow \Jet^\infty \Str_X$ is  $G$-equivariant.
The $G$-action gives rise to the homomorphism $\xi\mapsto \xi_{\Dsh},\g\rightarrow \Der(\Dsh),$ of Lie algebras.

For example, let $\Dcal$ be a $G$-equivariant quantization of $\Str_X$. The diagonal
$G$-action  on $\Str_X\otimes \Dcal$ extends to $(\Str_X\otimes \Dcal)^\wedge$ and so gives rise to a
$G$-action on $\Jet^\infty \Dcal$ making $\Jet^\infty \Dcal$ a $G$-equivariant quantum jet bundle.

Now let us introduce the notion of a {\it Hamiltonian} $G$-equivariant quantum jet bundle.
Suppose that there is a moment map $\mu:X\rightarrow \g^*$ for the $G$-action. For $\xi\in\g$ set $\Phi(\xi)=\pi_2^*(\mu^*(\xi))\in \Gamma(X,\Jet^\infty\Str_X)^\nabla$. We have the maps
$\xi\mapsto \xi_{X\times X},\xi^1_{X\times X},\xi^2_{X\times X}:\g\rightarrow \Der(\Str_{X\times X})$
associated with the diagonal $G$-action and the actions of $G$ on the first and on the second copy of
$X$, respectively. Clearly, $\xi_{X\times X}=\xi^1_{X\times X}+\xi^2_{X\times X}$. Therefore for the induced derivation
$\xi_{\Jet^\infty \Str_X}$ of the jet bundle we have $\xi_{\Jet^\infty \Str_X}=\nabla_{\xi_X}+ \{\mu^*(\xi),\cdot\}$.
We say that a $G$-equivariant quantum jet bundle $(\Dsh,\widetilde{\nabla}, \Theta)$ is Hamiltonian
if it is equipped with a map $\Phi_\hbar:\g\rightarrow \Gamma(X,\Dsh)^{\widetilde{\nabla}}$ such that
\begin{equation}\label{eq:Dsh_connection}
\xi_{\Dsh}=\widetilde{\nabla}_{\xi_X}+\frac{1}{\hbar^2}[\Phi_\hbar(\xi),\cdot].
\end{equation}
and $\Theta(\Phi_\hbar(\xi))=\Phi(\xi)$.
Now if $(\Dcal,\theta)$ is a Hamiltonian $G$-equivariant quantization with a quantum comoment
map $\varphi_\hbar$, then $\Jet^\infty \Dcal$ is Hamiltonian with $\Phi_\hbar(\xi)=\varphi_\hbar(\xi)\in \Gamma(X, \Jet^\infty \Dcal)^{\widetilde{\nabla}}$.

The quantum comoment map $\Phi_\hbar$ extends to a certain sheaf homomorphism
which will be important in the sequel. Namely, define the {\it graded}
universal enveloping algebra $\U_\hbar$ of $\g$ as the quotient of $T(\g)[\hbar]$ by the relations
$\xi\otimes \eta-\eta\otimes \xi-\hbar^2[\xi,\eta], \xi,\eta\in\g$. Then $\Phi_\hbar$ extends to an
algebra homomorphism $\U_\hbar\rightarrow \Gamma(X, \Dsh^{\widetilde{\nabla}})$. Extend $\Phi_\hbar$ to a homomorphism $\Str_X\otimes \U_\hbar\rightarrow \Dsh$ by $\Str_X$-linearity.
Further,
 we have a homomorphism $$\Str_X\otimes \U_\hbar/\hbar \Str_X\otimes \U_\hbar\rightarrow
\Str_X\otimes \K[X]$$ given by $f\otimes x\mapsto f\otimes \mu^*(x)$, where $f$ is a local section of
$\Str_X$ and $x\in S(\g)=\U_\hbar/\hbar \U_\hbar$. Let $I_{\mu,\Delta}$ denote the inverse
image of $I_\Delta$ in $\Str_X\otimes \U_\hbar/\hbar \Str_X\otimes\U_\hbar$ and $\widetilde{I}_{\mu,\Delta}$ be the inverse
image of $I_{\mu,\Delta}$ in $\Str_X\otimes \U_\hbar$. Consider the completion $$\Jet^\infty \U_\hbar:=\varprojlim_{k\rightarrow \infty}\Str_X\otimes \U_\hbar/(\Str_X\otimes \U_\hbar)\widetilde{I}_{\mu,\Delta}^k$$ The homomorphism $\Phi_\hbar:\Str_X\otimes \U_\hbar\rightarrow \Dsh$
is continuous in the $\widetilde{I}_{\mu,\Delta}$-adic topology and so extends to a continuous homomorphism
$\Jet^\infty \U_\hbar\rightarrow \Dsh$ (also denoted by $\Phi_\hbar$) in a unique way. We remark that $\Jet^\infty \U_\hbar$ comes equipped with a natural
connection and the homomorphism $\Phi_\hbar$ intertwines the connections.

Finally, let us define {\it graded} Hamiltonian $G$-equivariant quantum jet bundles.
We have a natural $\K^\times$-action on $\Jet^\infty \Str_X$. We say that a Hamiltonian
$G$-equivariant quantum jet bundle $(\Dsh,\widetilde{\nabla},\Theta)$ is graded,
if $\Dsh$ is equipped with a $\K^\times$-action by algebra automorphisms such that
\begin{itemize}
\item the action commutes with $G$.
\item $t.\hbar=t\hbar, t.\Phi_\hbar(\xi)=t^2\Phi_\hbar(\xi)$.
\item $\Theta$ is $\K^\times$-equivariant and $\widetilde{\nabla}$ is $\K^\times$-invariant.
\end{itemize}
If $\Dcal$ is graded, then $\Jet^\infty \Dcal$ has a natural $\K^\times$-action and is graded
with respect to this action.

The previous discussion implies the following corollary of Proposition \ref{Prop:jet1}.

\begin{Cor}\label{Cor:equiv}
The equivalences of Proposition \ref{Prop:jet1} define mutually inverse equivalences between
\begin{itemize}
\item the category of graded Hamiltonian $G$-equivariant quantizations and
\item the category of graded Hamiltonian $G$-equivariant quantum jet bundles.
\end{itemize}
\end{Cor}
\subsection{Existence of quantum comoment map}\label{SUBSECTION_qcm}
In this subsection we show that under some conditions (satisfied for all coverings of
$\Orb$) any $G$-equivariant quantization possesses a unique quantum comoment map.

The main result is the following proposition.

\begin{Prop}\label{Prop:qcm_exist}
Let $G$ be a semisimple algebraic group and $X$ be a symplectic variety.
Let $G$ act on $X$ by symplectomorphisms. Suppose  $H^1_{DR}(X)=\{0\}$.
Let $\Dcal$ be a $G$-equivariant quantization of $X$. Then there exists a
unique quantum comoment map $\varphi_\hbar: \g\rightarrow \Gamma(X,\Dcal)$.
\end{Prop}
%
%The proof is pretty standard but we provide it for reader's convenience.%
%
\begin{proof}
We remark that this is a standard fact that a moment map for the action of $G$ on $X$ exists
and is unique.

Now let us check that there is a linear map $\varphi'_\hbar:\g\rightarrow \Gamma(X,\Dcal)$
such that $\frac{1}{\hbar^2}[\varphi'_\hbar(\xi),\cdot]=\xi_{\Dcal}$. Then for $\varphi_\hbar$ we can take
the $G$-invariant component of $\varphi'_\hbar$.

It is enough to show that an element $f\in \Gamma(X,\Dcal)$ with $\frac{1}{\hbar^2}[f,\cdot]=\xi_\Dcal$
exists for any element $\xi\in\g$  lying  in the Lie algebra
of a one-dimensional torus, say $T$, of $G$.

There is a $T$-stable open affine covering  $X=\bigcup_i X_i$. The restriction of
$\Dcal$ to $X_i$ is, of course, a $T$-equivariant quantization of $X_i$. Therefore the restriction
is isomorphic to the Fedosov quantization of $X_i$ with a $T$-invariant curvature form, say, $\Omega_i$.
Recall that for $f_i\in \K[X_i][[\hbar]]$
the condition $\frac{1}{\hbar^2}[f,\cdot]=\xi_{X_i}$ is equivalent to $df_i=\iota_{\xi_X}\Omega_i$, see Proposition \ref{Prop_affine_qcm}.

The restrictions of $\Omega_i,\Omega_j$ to $X_i\cap X_j$ have the same cohomology class. So one can find  $T$-invariant
1-forms $\alpha_{ij}$ on $X_i\cap X_j$ such that $d\alpha_{ij}=\Omega_i-\Omega_j$ and $\alpha_{ij}=-\alpha_{ji}$.
Set $\beta_i:=\iota_\xi\Omega_i, g_{ij}:=\iota_\xi \alpha_{ij}$. Let us check that there are  $h_{ij}\in \K$ such that $\beta_i, g_{ij}+h_{ij}$ form
a De Rham 1-cocycle.

The equalities $d\beta_i=0, dg_{ij}=\beta_i-\beta_j$ follow from the condition that
$\Omega_i,a_{ij}$ are $T$-invariant and so vanish under the Lie derivative of $\xi$.
It remains to check that $h_{ijk}:=g_{ij}+g_{jk}+g_{ki}=0$. Remark that $d h_{ijk}=0$.
So $h_{ijk}$ is a 2-cocycle on $X$ (in the Zariski topology) with coefficients in $\K$.
Such a cocycle is  a coboundary (from the irreducibility of $X$): there are constants $h_{ij}$ with
$h_{ijk}=h_{ij}+h_{jk}+h_{ki}$. These are constants we need.

Since $H^1_{DR}(X)=\{0\}$, we see that $\beta_i, g_{ij}+h_{ij}$ is a coboundary and the
existence of $f_i$ follows.

Abusing the notation we  denote the element of $\Gamma(X_i,\Dcal)$ corresponding to $f_i$
again by $f_i$. We remark that $f_i-f_j$ lies in the center of $\Gamma(X_i\cap X_j,\Dcal)$.
The latter coincides with $\K[[\hbar]]$. So  $f_i-f_j$ form a 1-cocycle with coefficients
in $\K[[\hbar]]$. So we can add elements of $\K[[\hbar]]$ to the  $f_i$'s to  glue $f_i$ into
a global section $f$ of $\Dcal$.

We have just proved the existence of a quantum comoment map. As in the classical setting the uniqueness
follows from the fact that $\g$ has no nontrivial central extensions.
\end{proof}

\begin{Cor}\label{Cor:qcm_graded}
We preserve the conventions of Proposition \ref{Prop:qcm_exist}
Suppose $\K^\times$ acts on $X$ as in Subsection \ref{SUBSECTION_quant_generalities}
and $\Dcal$ is graded.
Then the quantum comoment map $\varphi_\hbar:\g\rightarrow \Gamma(X,\Dcal)$ satisfies
$t.\varphi_\hbar(\xi)=t^2\varphi(\hbar)$.
\end{Cor}
\begin{proof}
We remark that $\xi\mapsto t^{-2} (t.\varphi_\hbar(\xi))$ is again a quantum comoment map.
Now the claim follows from the uniqueness of the quantum comoment map.
\end{proof}

Now let $X$ be a covering of a nilpotent orbit $\Orb$ in $\g$. Since $X$ is a homogeneous space
of a simply connected semisimple algebraic group, we have $H^1_{DR}(X)=\{0\}$.
So any $G$-equivariant quantization of $X$ has a unique quantum comoment map.

\section{W-algebras}\label{SECTION_W}
In this section we recall some results about W-algebras.
In Subsection \ref{SUBSECTION_W_generalities} we define a W-algebra $\Walg$ following \cite{GG} (the definition
given there is very close to the original definition of Premet, \cite{Premet1}) and recall a category equivalence
proved by Skryabin in the appendix to \cite{Premet1}.

Subsection \ref{SUBSECTION_Decomp} is devoted to a basic result on W-algebras from \cite{Wquant},
the so called decomposition theorem. It says that a certain completion
$\U^\wedge_\hbar$ of the algebra $\U_\hbar$ decomposes into the completed tensor product
of the completed W-algebra and of a formal Weyl algebra. In Subsection \ref{SUBSECTION_correspondence} we use this result to establish a correspondence between the sets of ideals in $\Walg,\U,\U^\wedge_\hbar$ also originally
obtained in \cite{Wquant}.

Finally, in Subsection \ref{SUBSECTION_W_onedim} we recall some known results on 1-dimensional
$\Walg$-modules.

\subsection{Generalities}\label{SUBSECTION_W_generalities}
In this subsection we will sketch a definition of a W-algebra associated to $(\g,\Orb)$ (due to Premet, \cite{Premet1}).

Let $e\in \Orb$. Recall an $\sl_2$-triple $(e,h,f)$, the subgroup $Q:=Z_G(e,h,f)$ and a
homomorphism $\gamma:\K^\times\rightarrow G$ considered in Subsection \ref{SUBSECTION_quant_generalities}.

%Define a {\it Slodowy slice} $S$ by $S:=e+\z_\g(f)$. This is an affine subspace in $\g$ but it will
%be more convenient for us to consider it as a subspace in $\g^*$. Let $\chi\in \g^*$ be the element
%corresponding to $e$. Consider a {\it Kazhdan} $\K^\times$-action on $\g^*$ defined by
%$t.\alpha= t^{-2}\gamma(t)\alpha$. We remark that $\chi$ is $\K^\times$-invariant and that
%$\K^\times$ preserves $S$. Moreover, the $\K^\times$-action on $S$ is contracting: $\lim_{t\rightarrow %\infty}t.s=\chi$
%for any $\chi\in S$. So $\K[S]$ comes equipped with a positive grading $\K[S]=\bigoplus_{i\geqslant 0}\K[S]_i$
%with $\K[S]_0=\K$.

A W-algebra $\Walg$  can be defined as a quantum Hamiltonian reduction $$(\U/\U\m_\chi)^{\ad\m}:=\{a+ \U\m_\chi|
[\xi,a]\in \U\m_\chi, \forall \xi\in \m\},$$
where $\m\subset\g$ is a subalgebra, $\chi:\m\rightarrow \K$ is a character, both to be specified below, $\m_\chi:=\{\xi-\langle\chi,\xi\rangle,\xi\in \m\}$.

The subalgebra $\m\subset\g$ and the character $\chi$ are constructed as follows. Consider the grading
$\g=\bigoplus_{i\in \Z}\g(i)$, where $\g(i):=\{[h,\xi]=i\xi\}$. Set $\chi:=(e,\cdot)\in\g^*$ and consider
the skew-symmetric form $\omega_\chi(\xi,\eta)=\langle\chi,[\xi,\eta]\rangle$ on $\g$. The restriction of
$\omega_\chi$ to $\g(-1)$ is non-degenerate. Pick a lagrangian subspace $l\subset \g(-1)$ and
set $\m:=l\oplus\bigoplus_{i\leqslant -2}\g(i)$. The restriction of $\chi$ to $\m$ is a character of
$\m$.

The algebra $\Walg$ has a nice filtration called the Kazhdan filtration. It is inherited from
a {\it Kazhdan filtration} on $\U$ defined as follows.  Let $\F_k^{st}\U$ denote the standard PBW
filtration on $\U$ and $\U(j):=\{u\in \U| [h,u]=ju\}$. Then the Kazhdan filtration
$\KF_i \U$ on $\U$ is defined by $\KF_i\U:=\sum_{j+2k\leqslant i}\F_k^{st}\U\cap \U(j)$.
We have the induced filtrations $\KF_i (\U/\U\m_\chi),\KF_i\Walg$ on $\U/\U\m_\chi,\Walg$,
respectively. We remark that $\KF_0 (\U/\U\m_\chi)$ is spanned by the image of $1\in \U$.
It follows that $\KF_0\Walg$ is spanned by the unit of $\Walg$.

The associated graded algebra of $\Walg$ has the following nice description.
Define a {\it Slodowy slice} $S$ by $S:=e+\z_\g(f)$. This is an affine subspace in $\g$ but it will
be more convenient for us to consider it as a subspace in $\g^*$.
Consider a {\it Kazhdan} $\K^\times$-action on $\g^*$ defined by
$t.\alpha= t^{-2}\gamma(t)\alpha$, where $\gamma:\K^\times\rightarrow G$
was defined in Subsection \ref{SUBSECTION_quant_generalities}. We remark that $\chi$ is $\K^\times$-invariant and that
$\K^\times$ preserves $S$. Moreover, the $\K^\times$-action on $S$ is contracting: $\lim_{t\rightarrow \infty}t.s=\chi$
for any $s\in S$. So $\K[S]$ comes equipped with a positive grading $\K[S]=\bigoplus_{i\geqslant 0}\K[S]_i$
with $\K[S]_0=\K$. As Premet proved in \cite{Premet1}, $\gr\Walg\cong \K[S]$.

The $S(\g)$-module $\gr \U/\U\m_\chi$ also has a nice description, see \cite{GG}. Namely,
a natural homomorphism $S(\g)/S(\g)\m_\chi\rightarrow \gr \U/\U\m_\chi$ is a bijection.
Also, as was shown by Gan and Ginzburg, there is a natural identification
$S(\g)/S(\g)\m_\chi\cong \K[M]\otimes \K[S]$, where $M$ is the unipotent subgroup
of $G$ with Lie algebra $\m$. This identification preserves the gradings, where the
grading on $\K[M]$ is induced by the $\K^\times$-action $(t,m)\mapsto \gamma(t)m\gamma(t)^{-1},
t\in \K^\times, m\in M$.

To finish the subsection let us recall the Skryabin equivalence, see the appendix to \cite{Premet1}.
This is an equivalence between the category $\Walg$-$\operatorname{Mod}$ of left $\Walg$-modules and the category
 $\operatorname{Wh}$ of {\it Whittaker}
$\U$-modules. By definition, a  left $\U$-module $M$ is called {\it Whittaker}
if $\m_\chi$ acts on $M$ by locally nilpotent endomorphisms. According to Skryabin,
the functor $N\mapsto (\U/\U\m_\chi)\otimes_{\Walg}N$ is an equivalence between the categories $\Walg$-$\operatorname{Mod}$ and $\operatorname{Wh}$. A quasiinverse equivalence sends $M\in \operatorname{Wh}$ to
the space $M^{\m_\chi}$ of $\m_\chi$-invariants.

Now let $N$ be a finitely generated $\Walg$-module. Equip $N$ with a filtration that is compatible with the Kazhdan filtration on $\Walg$. Assume in addition that $\gr N$ is a finitely generated $\K[S]$-module.
Then one can equip $\U/\U\m_\chi\otimes_{\Walg} N$ with the product filtration. The following result was obtained
in the proof of Theorem 6.1 in \cite{GG}.

\begin{Lem}\label{Lem:GG_filtr}
The natural homomorphism $$\K[M]\otimes \gr N=(\gr \U/\U\m_\chi)\otimes_{\K[S]}\gr N\rightarrow \gr \left((\U/\U\m_\chi)\otimes_{\Walg}N\right)$$ is an isomorphism.
\end{Lem}

%A W-algebra $\Walg$ associated with the pair $(\g,\Orb)$ was defined in the full generality
%in \cite{Premet1}, an alternative definition was given in \cite{Wquant}, Subsection 3.1. We will
%need the following properties. For the proofs a reader is referred to \cite{Wquant}, Subsection 3.1,
%and \cite{HC}, Subsection 2.2.

%1) $\Walg$ is an associative unital algebra with a {\it Kazhdan filtration $\F_i\Walg$} and $\gr\Walg$
%s naturally identified with $\K[S]$.
%
%2) The group $Q$ acts on $\Walg$ preserving the filtration and the identification  $\gr\Walg\cong \K[S]$
%is $Q$-equivariant. Moreover, the Lie algebra $\q$ is $Q$-equivariantly embedded into $\F_2\Walg$ and its
%adjoint action coincides with the differential of the $Q$-action.

\subsection{Decomposition theorem}\label{SUBSECTION_Decomp}

The most crucial property of W-algebras we need is the decomposition theorem, see \cite{Wquant}, Subsection 3.3,
and \cite{HC}, Subsection 2.3. This theorem asserts that, up to a suitably understood completion,
the universal enveloping algebra $\U$ of $\g$ is decomposed into the tensor product of the
W-algebra and of certain Weyl algebra. It will be convenient for us to work with "homogeneous"
versions of our algebras.

Equip $\U$ with the "doubled" standard filtration $\F_i\U$, where $\F_i\U$ is spanned by all monomials
$\xi_1\ldots\xi_k$ with $2k\leqslant i, \xi_1,\ldots,\xi_k\in \g$. Then form the Rees algebra
$R_\hbar(\U)=\bigoplus_{i\geqslant 0}\hbar^i \F_i\U\subset \U[\hbar]$. This Rees algebra is naturally isomorphic to
the algebra $\U_\hbar$ introduced in Subsection \ref{SUBSECTION_qjb}. We have a natural identification
$\U_\hbar/\hbar\U_\hbar=S\g=\K[\g^*]$. Let $I_\chi$ denote the maximal ideal of $\chi$ in $S\g$
and $\widetilde{I}_{\chi}$ be its preimage  under the natural projection $\U_\hbar\twoheadrightarrow S\g$.

A completion of $\U_\hbar$ we need is $\U^\wedge_\hbar:=\varprojlim_{k\rightarrow \infty} \U/\widetilde{I}_{\chi}^k$.
As we have seen in \cite{HC}, Subsection 2.4, $\U^\wedge_\hbar$ can be considered as the space
$\K[\g^*]^\wedge_\chi[[\hbar]]$ equipped with a new (deformed) product, here $\K[\g^*]^\wedge_\chi:=\varprojlim_{k\rightarrow \infty} \K[\g^*]/I_\chi^k$ is the algebra of formal
power series in the neighborhood of $\chi$.

We will need two group actions on $\U^\wedge_\hbar$. Let $Q:=Z_G(e,h,f)$ be the centralizer of
$(e,h,f)$ in $G$. There is a natural $Q$-action on $\U_\hbar$ by graded algebra automorphisms.
This action stabilizes $\widetilde{I}_{\chi}$ and so uniquely extends to a $Q$-action on $\U^\wedge_\hbar$
by topological algebra automorphisms.
 Also there is a Kazhdan action of $\K^\times$ on $\U^\wedge_\hbar$ defined as follows:  an element $t\in \K^\times$ acts on $\hbar^i \F_i\U$ via  $(t,u)\mapsto t^i\gamma(t)u$. This action fixes
$\widetilde{I}_{\chi}$ and again lifts to $\U^\wedge_\hbar$.

Next, consider the homogeneous version $\Walg_\hbar=R_\hbar(\Walg)$ of the W-algebra. Similarly to the previous
paragraph, define the completion $\Walg^\wedge_\hbar$ of $\Walg_\hbar$ with respect to the maximal ideal
of $\chi\in S$. There is a $\K^\times$-action on $\Walg_\hbar$ given by $t.\hbar^i w=t^i \hbar^i w$
that can again can be naturally extended to $\Walg^\wedge_\hbar$.
An important difference of this action from that
on $\U_\hbar^\wedge$ is that $\Walg_\hbar$ coincide with the subalgebra of $\Walg_\hbar^\wedge$
consisting of all $\K^\times$-finite (locally finite in the terminology of \cite{HC}) elements.

The group $Q$ also acts on $\Walg^\wedge_\hbar$ as follows. There is a $Q$-action on $\Walg$
by filtered algebra automorphisms, see \cite{Premet2}, 1.2 or \cite{Wquant}, Subsections 3.1,3.3.
This action gives rise to a $Q$-action on $\Walg_\hbar$ and the latter extends to
$\Walg_\hbar^\wedge$.

Finally, we need a completed version of the Weyl algebra of an appropriate symplectic vector space.
A vector space we need is $V:=\im \ad(f)$. The restriction of $\omega_\chi$ to $V$ is non-degenerate.
We equip $V$ with the $Q$-action restricted from the coadjoint $G$-action and with the $\K^\times$-action given by $t.v=\gamma(t)^{-1}v$. In particular, $t\in \K^\times$ multiplies $\omega_\chi$ by $t^2$.

By definition, the  {\it homogeneous Weyl algebra} $\W_\hbar$ of $V$ is the quotient of $T(V)[\hbar]$
by $u\otimes v-v\otimes u -\hbar^2 \omega_\chi(u,v), u,v\in V$. Let $\W_\hbar^\wedge$ denote the completion of
$\W_\hbar$ at 0.

The algebras $\U^\wedge_\hbar, \Walg^\wedge_\hbar, \W^\wedge_\hbar$ have  natural topologies: the
topologies of inverse limits.

%Also we remark that we have Lie algebra homomorphisms $\q\rightarrow \U_\hbar,\Walg_\hbar,\W_\hbar$
%(with respect to the brackets on the three algebras are given by $[a,b]=\frac{1}{\hbar^2}(ab-ba)$).
%The homomorphism $\q\rightarrow \U_\hbar$ comes from the inclusion $\q\hookrightarrow \g$.
%The homomorphism $\q\rightarrow \Walg_\hbar$ comes from the homomorphism $\q\rightarrow \Walg$.
%Finally, a homomorphism $\q\rightarrow \W_\hbar$ is the composition of the homomorphism
%$\q\rightarrow \mathfrak{sp}(V)$ induced by the action of $Q$ on $V$ and the natural embedding
%of $\mathfrak{sp}(V)$ into the component of $\W_\hbar$ of degree 2.

Consider the completed tensor product $\W_\hbar^\wedge(\Walg^\wedge_\hbar):=\W_\hbar^\wedge \widehat{\otimes}_{\K[[\hbar]]} \Walg^\wedge_\hbar$ (i.e., we first take the usual tensor
product of the topological algebras and then complete it with respect to the induced topology).
%Abusing the notation, write $\widetilde{I}_{\chi}$ for the maximal ideal of $\U^\wedge_\hbar$
%and $\m'_{\chi,\hbar}$ for the maximal ideal in $\W^\wedge_\hbar(\Walg^\wedge_\hbar)$. The space %$\widetilde{I}_{\chi}/(\widetilde{I}_{\chi}^2+(\hbar))$ is identified with the cotangent space to $\g^*$
%at $\chi$, i.e., with $\g$. Similarly, $\m'_{\chi,\hbar}/(\m'^2_{\chi,\hbar}+(\hbar))$ is naturally identified
%with $\ker\ad(e)+\im \ad(f)$ (for $\ker\ad(e)$ is the cotangent space to $S\subset \g^*$ at
%$\chi$). In particular, we get an identification %\begin{equation}\label{eq:map}\m'_{\chi,\hbar}/(\m'^2_{\chi,\hbar}+(\hbar))\rightarrow
%\widetilde{I}_{\chi}/(\widetilde{I}_{\chi}^2+(\hbar)), (\xi,\eta)\mapsto \xi+\eta.\end{equation}%
%
The decomposition theorem is the following statement.

\begin{Prop}\label{Prop:decomposition}
There is a $Q\times \K^\times$-equivariant  isomorphism $\Psi_\hbar: \U^\wedge_\hbar\rightarrow \W_\hbar^\wedge(\Walg_\hbar^\wedge)$
of topological $\K[[\hbar]]$-algebras.
\end{Prop}

For the proof see \cite{Wquant}, Theorem 3.3.1 and the discussion after the theorem.

In the sequel we identify $\U^\wedge_\hbar$ with $\W_\hbar^\wedge(\Walg_\hbar^\wedge)$ by means
of $\Psi_\hbar$.

\subsection{Correspondence between ideals}\label{SUBSECTION_correspondence}
We  need to relate two-sided ideals in $\U^\wedge_\hbar$ and $\Walg$. More precisely, we consider
two sets: the set $\Id(\Walg)$ of two-sided ideals in $\Walg$ and the set $\Id_\hbar(\U^\wedge_\hbar)$
consisting of all $\K^\times$-stable $\hbar$-saturated two-sided ideals in $\U^\wedge_\hbar$ (an ideal
$\I'_\hbar\subset \U^\wedge_\hbar$ is said to be $\hbar$-saturated if $\hbar a\in \I'_\hbar$
implies $a\in \I'_\hbar$, equivalently, if the quotient $\U^\wedge_\hbar/\I'_\hbar$ is
a flat $\K[[\hbar]]$-module). Construct a map $\I\mapsto \I^\#:\Id_\hbar(\Walg)\rightarrow \Id_\hbar(\U^\wedge_\hbar)$ is as follows. Pick $\I\in \Id(\Walg)$. Form the Rees ideal
$\I_\hbar:=\bigoplus_{i\geqslant 0} (\F_i\Walg\cap \I)\hbar^i$. Then take the closure
$\I^\wedge_\hbar\subset \Walg^\wedge_\hbar$ of $\I_\hbar$. We set
$\I^\#:=\W^\wedge_\hbar\widehat{\otimes}_{\K[[\hbar]]}\I^\wedge_\hbar\subset \W^\wedge_\hbar(\Walg^\wedge_\hbar)=\U^\wedge_\hbar$.

\begin{Lem}\label{Lem:maps_btw_ideals2}
The map $\I\mapsto \I^\#:\Id(\Walg)\rightarrow \Id_\hbar(\U^\wedge_\hbar)$ is a bijection.
The inverse map sends $\J'_\hbar\in \Id_\hbar(\U^\wedge_\hbar)$ to the image, denoted by $\J'_{\hbar\#}$, of
$\J'_\hbar\cap \Walg_\hbar$ under the natural epimorphism $\Walg_\hbar\twoheadrightarrow \Walg$.
\end{Lem}

This is proved in \cite{HC}, Proposition 3.3.1.

\begin{Rem}\label{Rem:ideal_stability}
It follows from the definition that the map $\I\mapsto \I^{\#}$ is $Q$-equivariant. We remark that
any element of $\Id_\hbar(\U^\wedge_\hbar)$ is $Z_G(e)^\circ$-stable because the differential of the
$Z_G(e)$-action on $\U^\wedge_\hbar$ is given by $\xi\mapsto \frac{1}{\hbar^2}[\xi,\cdot]$. In particular,
it follows that any element of $\Id(\Walg)$ is $Q^\circ$-stable. So we see that
the component group $C(e):=Q/Q^\circ=Z_G(e)/Z_G(e)^\circ$ acts on both $\Id(\Walg)$ and $\Id_\hbar(\U^\wedge_\hbar)$.
\end{Rem}

The following lemma follows directly from the construction of the map $\I\mapsto \I^{\#}$.

\begin{Lem}\label{Lem:maps_btw_ideals}
$\dim \Walg/\I=1$ if and only if $(\U^\wedge_\hbar/\I^\#)\hbar(\U^\wedge_\hbar/\I^\#)$
coincides with the completion $\K[\Orb]^\wedge_\chi$ of $\K[\Orb]$ at $\chi$.
\end{Lem}

Following \cite{Wquant}, define a map $\bullet^\dagger:\Id(\Walg)\rightarrow \Id(\U)$ by
setting $\I^\dagger$ to be the image of $\I^{\#}\cap\U_\hbar$ in $\U$. The map $\I\mapsto
\I^\dagger$ has the following properties (see \cite{Wquant}, Theorem 1.2.1 and \cite{HC}, Conjecture 1.2.1).

\begin{Prop}\label{Prop:up_dag}
\begin{enumerate}
\item If $\I$ is primitive, then so is $\I^\dagger$.
\item If $\dim \Walg/\I<\infty$, then the associated variety $\VA(\U/\I^\dagger)$
coincides with $\Orb$.
\item Let $\J$ be a primitive ideal in $\U$ with $\VA(\U/\J)=\overline{\Orb}$.
Then the set of $\I\in \Id(\Walg)$ with $\dim \Walg/\I<\infty$ and $\I^\dagger=\J$ is a single
$C(e)$-orbit.
\end{enumerate}
\end{Prop}

%Now we are ready to produce maps between $\Id(\U)$ (=the set of two-sided ideals in $\U$) and $\Id(\Walg)$.
%
%Let $\J\subset \U$ be a two-sided ideal. Form the Rees ideal
%$\J_\hbar:=\bigoplus_{i\geqslant 0} (\F_i\U\cap \J)\hbar^i$. Then take its completion
%$\J^\wedge_\hbar\subset \U^\wedge_\hbar$. This is an element of $\Id_\hbar(\U^\wedge_\hbar)$.
%Let $\J_\dagger$ be the corresponding ideal of $\Walg$.
%
%Let us construct a map in the opposite direction. Pick $\I\subset \Id(\Walg)$.
%Let $\J'_\hbar$ denote the corresponding ideal in $\U^\wedge_\hbar$. Take the intersection
%$\J_\hbar:=\J'_\hbar\cap\U_\hbar$ and let $\I^\dagger$ denote the image of
%$\J_\hbar$ in $\U$.
%
%Consider the subsets $\Id_\Orb(\U)\subset \Id(\U)$ consisting of all ideals $\J$
%with $\VA(\U/\J)=\overline{\Orb}$ and $\Id_{fin}(\Walg)\subset \Id(\Walg)$
%consisting of all ideals of finite codimension. The maps between $\Id(\U),\Id(\Walg)$
%constructed above restrict to maps between $\Id_{\Orb}(\U),\Id_{fin}(\Walg)$.
%These maps enjoy the following properties, see Theorems 1.2.2 and 3.1.1 in \cite{HC}.
%
%\begin{Prop}\label{Prop:ideals}
%\begin{enumerate}
%\item For $\J\in \Id_{\Orb}(\U)$ one has $\J\subset (\J_\dagger)^{\dagger}$ and
%$\VA((\J_\dagger)^\dagger/\J)\subset \overline{\Orb}\setminus \Orb$.
%\item For a {\rm $Q$-stable} element $\I\in \Id_{fin}(\Walg)$ one has
%$(\I^\dagger)_{\dagger}=\I$.
%\item Let $\J\in \Id_\Orb(\U)$ be primitive. Then the set $\{\I\in \Id_{fin}(\Walg)| \I^\dagger=\J\}$
%is a  single $C(e)$-orbit.
%\end{enumerate}
%\end{Prop}

\subsection{1-dimensional representations}\label{SUBSECTION_W_onedim}
In this subsection we will explain known results about 1-dimensional representations of W-algebras.
Let $\Id^1(\Walg)$ denote the set of two-sided ideals of codimension 1 in $\Walg$.
We start with the existence theorem.

\begin{Thm}\label{Thm:one_dim_1}
\begin{itemize}
\item[(i)] If $\g$ is classical, then $\Id^1(\Walg)^{C(e)}\neq \varnothing$.
\item[(ii)] If $\g$ is $G_2,F_4,E_6,E_7$ and $e$ is arbitrary, or $\g$ is $E_8$ and $e$ is not rigid, then
$\Id^1(\Walg)\neq \varnothing$.
\end{itemize}
\end{Thm}

The first assertion was proved in \cite{Wquant}, Theorem 1.2.3. However, most ingredients there were not essentially new. The construction of an ideal $\J\subset \U$ such that $\Orb$ is open in $\VA(\U/\J)$
and the multiplicity of $\U/\J$ on $\Orb$ is 1 given there was first obtained by Brilynski in \cite{Brylinski}.
That such an ideal has the form $\I^\dagger$ for $\dim \Walg/\I=1$  was essentially observed by Moeglin in
\cite{Moeglin2}. Finally, the claim that $\I$ must be $C(e)$-invariant follows from \cite{Wquant}, Theorem 1.2.2.

Assertion (ii) follows from \cite{GRU} and \cite{Premet4}. According to \cite{GRU},
a one-dimensional $\Walg$-module exists for all rigid elements in $G_2,F_4,E_6,E_7$
and some (relatively small) rigid elements in $E_8$.
Here the term "rigid"  refers to the Lusztig-Spaltenstein induction,
see \cite{LS}. The main result of \cite{Premet4} is that the existence of a one-dimensional
module for a W-algebra is preserved by the induction. Hence (ii). The relation between 1-dimensional modules
and the Lusztig-Spaltenstein induction is also discussed in \cite{Miura}.

In fact, for a rigid element $e$ one can describe the set of ideals $\I^\dagger$ with $\dim \Walg/\I=1$
in terms of highest weights. Fix a Cartan subalgebra $\h\subset\g$ and a system $\Pi$ of simple roots. Recall that according to Duflo, every primitive ideal
in $\U$ has the form $J(\lambda):=\Ann_\U L(\lambda)$, where $L(\lambda)$ denotes the irreducible highest weight
module with highest weight $\lambda$ (we do not use the $\rho$-shift here, so $L(0)$ stands the trivial
one-dimensional module). An important remark is that $\lambda$ is not recovered from $J(\lambda)$ uniquely.

In \cite{Miura} we have found some conditions on $\lambda$ such that the ideals $J(\lambda)$
exhaust the set of ideals $\I^\dagger$ with $\dim\Walg/\I=1$ under the condition that the algebra
$\q=\z_\g(e,h,f)$ is {\it semisimple} that is always the case for rigid elements.
These conditions are (in a sense) combinatorial
provided $e$ is of principal Levi type.

In more detail, let $\t$ be a Cartan subalgebra in $\q$. Set $\lf:=\z_\g(\t)$
and let $L$ be the Levi subgroup of $G$ with Lie algebra $\lf$. Conjugating the triple
$(e,h,f)$ one may assume that $\h\subset\lf$ and that $\t$ contains a dominant element.
Let $\Walg^0$ denote the W-algebra constructed for the pair $(\lf,e)$. For this W-algebra
we have a map $\bullet^{\dagger_0}$ from the set of primitive ideals of finite codimension in $\Walg^0$
to the set of primitive ideals $\J_0\subset \U^0:=U(\lf)$ with $\VA(\U^0/\J_0)=\overline{Le}$.

We need a certain element $\delta\in \h^*$. Let $\Delta^{<0}$ denote the set of all negative
roots in $\Delta$ that are not roots of $\lf$. %For an integer $i$ set $\Delta^{<0}_i:=\{\alpha\in \Delta^{<0}|
%\langle\alpha,h\rangle=i\}$.
Set
\begin{equation}\label{eq:delta}
\delta:=\sum_{\alpha\in \Delta^{<0}, \langle\alpha,h\rangle=1}\frac{1}{2}\alpha+\sum_{\alpha\in \Delta^{<0},\langle\alpha,h\rangle\geqslant 2}\alpha.
\end{equation}

In \cite{Miura}, Subsection 5.3, we have proved the following result.

\begin{Cor}\label{Cor:1dim}
Suppose $\q$ is semisimple.
\begin{enumerate}
\item Let $\lambda\in \h^*$ satisfy the following four conditions:
\begin{itemize}
\item[(A)]  $\VA(\U^0/J_0(\lambda))=\overline{Le}$.
\item[(B)] $\dim \VA(\U/J(\lambda))\leqslant \dim \Orb$.
\item[(C)] $\lambda-\delta$ vanishes on the center $\z(\lf)$ of $\lf$.
\item[(D)] $J_0(\lambda)=\I_0^{\dagger_0}$ for some ideal $\I_0$ of codimension 1 in $\Walg^0$.
\end{itemize}
Then $J(\lambda)=\I^\dagger$ for some ideal $\I\subset\Walg$ of codimension 1.
\item For any $\I\subset \Walg$ of codimension 1 there is $\lambda\in \h^*$ satisfying
(A)-(D) and such that $J(\lambda)=\I^\dagger$.
\end{enumerate}
\end{Cor}
When $e$ is principal in $\lf$ the condition (A) means that $\lambda$ is antidominant for $\lf$,
while the condition (D) becomes vacuous. The condition (B) is still very difficult to check.

\section{Quantizations of nilpotent orbits}\label{SECTION_proofs}
This section is the main part of the paper. In Subsection \ref{SUBSECTION_Proof}
we prove Theorem \ref{Thm:main1}. In Subsection \ref{SUBSECTION_global_sections} we give a description of the algebra
of global sections of a quantization $\Dcal$ of $X$. Finally, in Subsection \ref{SUBSECTION_Moeglin}
we compare our results with Moeglin's, \cite{Moeglin2}.

\subsection{Proof of Theorem \ref{Thm:main1}}\label{SUBSECTION_Proof}
Recall that $G$ is a simply connected semisimple algebraic group, and $\g$ is its Lie algebra.

In this subsection we consider a $G$-equivariant covering $X$ of a nilpotent orbit $\Orb\subset \g\cong\g^*$.
Let $\Walg$ denote the W-algebra associated to $\Orb$.
Pick a point $\chi\in \Orb$ and a point $x\in X$ lying over $\chi$. Set $H:=G_x, \Gamma:=H/H^\circ\subset C(e)$.
We will construct mutually inverse bijections between the set of isomorphism classes of homogeneous Hamiltonian
$G$-equivariant quantum jet bundles and the set $\Id^1(\Walg)^\Gamma$.
Thanks to Corollary \ref{Cor:equiv} and the results of Subsection \ref{SUBSECTION_qcm}, this will imply
Theorem \ref{Thm:main1}.

Recall the flat sheaf $\Jet^\infty \U_\hbar$ on $X$.  We start with a standard lemma describing various properties
of $\Jet^\infty \U_\hbar$.

\begin{Lem}\label{Lem:Ush}
\begin{enumerate}
\item The fiber of $\Jet^\infty \U_\hbar$ at $x$ is naturally identified with $\U^\wedge_\hbar$.
\item Let $\Id(\Jet^\infty \U_\hbar)$ denote the set of homogeneous $G$-stable $\hbar$-saturated ideals in
$\Jet^\infty \U_\hbar$. Taking the fiber of an ideal at $x$ defines a bijection between $\Id(\Jet^\infty \U_\hbar)$
and $\Id(\U_\hbar^\wedge)^{\Gamma}$. The inverse map is given by $\J'_\hbar\mapsto \pi_*(\Str_G\widehat{\otimes} \J'_\hbar)^H$, where $\pi$ stands for the projection $G\twoheadrightarrow G/H$.
\item Any element of $\Id(\Jet^\infty \U_\hbar)$ is stable with respect to the connection $\widetilde{\nabla}$.
\end{enumerate}
\end{Lem}
\begin{proof}
Assertion (1) follows from the observation that the completion functor is right exact.

Let us proceed to the proof of (3).
Recall the equality $\xi_{\Jet^\infty \U_\hbar}=\widetilde{\nabla}_{\xi_X}+\frac{1}{\hbar^2}[\xi,\cdot],\xi\in\g$. Let
$\Ish\in \Id(\Jet^\infty \U_\hbar)$. Being an $\hbar$-saturated two-sided ideal, $\Ish$ is stable with respect to
$\frac{1}{\hbar^2}[\xi,\cdot]$. Being $G$-stable, $\Ish$ is stable with respect to $\xi_{\Jet^\infty \U_\hbar}$.
 So $\Ish$ is $\widetilde{\nabla}_{\xi_X}$-stable.
But the vector fields $\xi_{X}$ generated the tangent sheaf of $X$. So $\Ish$ is $\widetilde{\nabla}$-stable.

Let us prove (2). First of all, we recall that $\Jet^\infty \U_\hbar$ is a  $G$-equivariant pro-coherent sheaf
of $\Str_X$-algebras. Consider the category of all such algebras. Then the functor of taking the fiber
at $x$ defines an equivalence between this category and  the category of $H$-equivariant
pro-finite dimensional algebras. A quasiinverse equivalence is $\mathcal{A}\rightarrow \pi_*(\Str_G\widehat{\otimes} \A)^H$.

It remains to prove that $\Jet^\infty \U_\hbar/\Ish$ is pro-coherent for any $\Ish\in \Id(\Ush)$. This will follow if we check
that $\Ish$ is closed in $\Jet^\infty \U_\hbar$.
But any left ideal in $\Jet^\infty \U_\hbar$ is closed, compare with Lemma 2.4.4 in \cite{HC}, this lemma generalizes
to the sheaf setting directly.
\end{proof}

Now let $\Dsh$ be a homogeneous Hamiltonian $G$-equivariant quantum jet bundle on $X$.

\begin{Lem}\label{Lem:surjectivity}
The morphism $\Phi_\hbar:\Jet^\infty \U_\hbar\rightarrow \Dsh$ is surjective.
\end{Lem}
\begin{proof}
Both $\Jet^\infty \U_\hbar$ and $\Dsh$ are $G$-equivariant and $\Str_X$-pro-coherent. Therefore it is enough to show that
the induced homomorphism of fibers at $x$ is surjective. Both fibers are complete and separated in the
$\hbar$-adic topology. Therefore it remains to prove the surjectivity modulo $\hbar$.
Here we have the homomorphism $\K[\g^*]^{\wedge}_\chi\rightarrow \K[X]^\wedge_x$ of the completions
induced by the comoment map $\mu^*:\K[\g^*]\rightarrow \K[X]$. But $\mu$ is just the composition
of the covering $X\twoheadrightarrow \Orb$ and the inclusion $\Orb\hookrightarrow \g^*$.
So $\mu$ is unramified.  Hence the surjectivity claim.
\end{proof}

Let $\Ish$ denote the kernel of $\Phi_\hbar$. This is a homogeneous $G$-stable $\hbar$-saturated
(and automatically closed) ideal in $\Jet^\infty\U_\hbar$. Let $\I'_\hbar$ be the fiber of $\Ish$ at $x$. Then $\I'_\hbar$ is
a homogeneous $H$-stable $\hbar$-saturated ideal in $\U^\wedge_\hbar$. Then, by the construction, $(\U_\hbar^\wedge/\I'_\hbar)/\hbar(\U_\hbar^\wedge/\I'_\hbar)=\K[\Orb]^\wedge_\chi$. Set $\I_{\Dsh}:=\I'_{\hbar\#}$.  By  Lemma \ref{Lem:maps_btw_ideals}, $\dim \Walg/\I_{\Dsh}=1$. The inclusion $\I_{\Dsh}\in \Id(\Walg)^\Gamma$ follows from Remark
\ref{Rem:ideal_stability}.
 So we have got a map in one direction.

Let us describe a map in the opposite direction. Pick $\I\in \Id(\Walg)^\Gamma$.
Remark \ref{Rem:ideal_stability} implies $\I^{\#}$ is $H$-stable.
Also there is a natural isomorphism $\theta:(\U^\wedge_\hbar/\I^{\#})/\hbar (\U^\wedge_\hbar/\I^{\#})
\xrightarrow{\sim} \K[\Orb]^\wedge_\chi=\K[X]^\wedge_x$. Now let $\Ish$ be the ideal in
$\Jet^\infty \U_\hbar$ corresponding to $\I^{\#}$. Set $\Dsh_\I:=\Jet^\infty \U_\hbar/\Ish$.
In other words, $\Dsh_\I=\pi_*(\Str_G\widehat{\otimes} \U^\wedge_\hbar/\I^{\#})^H$.
By assertion (3) of Lemma  \ref{Lem:Ush}, $\Ish$ is $\widetilde{\nabla}$-stable.
It follows the sheaf $\Dsh_\I$ comes equipped with a flat
connection $\widetilde{\nabla}$ induced from the connection on $\Jet^\infty \U_\hbar$. Since
$\Jet^\infty \Str_X=\pi_*(\Str_G\widehat{\otimes} \K[X]^\wedge_x)^H$, we see that $\theta$
gives rise to an isomorphism $\Theta: \Dsh_\I/\hbar \Dsh_\I\xrightarrow{\sim}\Jet^\infty \Str_X$.
It is straightforward to verify that $(\Dsh_\I,\widetilde{\nabla},\Theta)$ is a homogeneous
Hamiltonian $G$-equivariant quantum jet bundle.

Also it is clear that the maps $\Dsh\mapsto \I_{\Dsh}, \I\mapsto \Dsh_\I$
are mutually inverse. This completes the proof of Theorem \ref{Thm:main1}.

\subsection{Global sections}\label{SUBSECTION_global_sections}
 Let $\Dcal$ be a homogeneous Hamiltonian $G$-equivariant quantization of
$X$. The goal of this subsection is to describe the algebra $\Gamma(X,\Dcal)$ of global sections.

Let $\I'_\hbar\in \Id_\hbar(\U^\wedge_\hbar)$. We say that an element $a\in \U^\wedge_\hbar/\I'_\hbar$
is  {\it finite}  if it lies in a finite dimensional $\g$- and $\K^\times$-stable subspace.
It is clear that all finite elements form a subalgebra in $\U^\wedge_\hbar/\I'_\hbar$.
We denote this subalgebra by $(\U^\wedge_\hbar/\I'_\hbar)_{fin}$.

Let $x\in X,\Gamma\subset C(e)$ be as above. Let us introduce a $\Gamma$-action on $(\U^\wedge_\hbar/\I'_\hbar)_{fin}$. The last algebra has two
$H$-actions: the action $\rho$ induced from the $H$-action on $\U^\wedge_\hbar/\I'_\hbar$
and the action $\rho'$ restricted from the $G$-action on $(\U^\wedge_\hbar/\I'_\hbar)_{fin}$.
Similarly to \cite{HC}, Subsection 3.2, $\rho\circ\rho'^{-1}$ descends to a $\Gamma=H/H^\circ$-action on $(\U^\wedge_\hbar/\I'_\hbar)_{fin}$ commuting with $G\times \K^\times$.

The main result is as follows.

\begin{Prop}\label{Prop:sections}
Let $\I'_\hbar(=\I_{\Dcal}^{\#})$ be the ideal in $\U_\hbar^\wedge$ corresponding to $\Dcal$.
The algebra $\Gamma(X,\Dcal)$  is naturally identified with
the $\hbar$-adic completion of $(\U^\wedge_\hbar/\I'_\hbar)_{fin}^\Gamma$.
\end{Prop}
\begin{proof}
% The description is carried over in terms of the ideal $\I'_\hbar\in \U^\wedge_\hbar$
%corresponding to $\Dcal$ via the bijections described above. Namely, we will see below that
%$\Gamma(X,\Dcal)$ is naturally isomorphic to the algebra of $\Gamma$-invariant $\ad\g$-locally finite vectors in
%$\U^\wedge_\hbar/\I'_\hbar$.
First of all let us produce an algebra homomorphism $\Gamma(X,\Dcal)\rightarrow \U^\wedge_\hbar/\I'_\hbar$.
The algebra $\Gamma(X,\Dcal)$ coincides the algebra $\Gamma(X,\Dsh)^{\widetilde{\nabla}}$ of the global flat section
of $\Dsh:=\Jet^\infty \Dcal$. By the construction of the previous subsection, $\U^\wedge_\hbar/\I'_\hbar$
is the fiber $\Dsh_x$ of $\Dsh$ at $x$. A homomorphism we need is $\Gamma(X,\Dcal)\hookrightarrow \Gamma(X,\Dsh)\twoheadrightarrow \Dsh_x=\U^\wedge_\hbar/\I'_\hbar$.

Let us check that this homomorphism is injective. Let $K$ stand for the kernel.  It follows from
(\ref{eq:Dsh_connection}) that on $\Gamma(X,\Dcal)$
the derivation $\xi_{\Dsh}$ coincides with $\frac{1}{\hbar^2}[\Phi_\hbar(\xi),\cdot]$.
Being an $\hbar$-saturated two-sided ideal in $\Gamma(X,\Dcal)$, the kernel $K$ is $\xi_{\Dsh}$-stable
for any $\xi\in\g$. This means that $K$ is $G$-stable. Therefore any element of $K\subset \Gamma(X,\Dsh)$
vanishes in every point of $X$. So $K=\{0\}$.

The subalgebra $\Gamma(X,\Dcal)_{fin}$ of finite elements of  $\Gamma(X,\Dcal)$ is dense
in the $\hbar$-adic topology. So it remains to check that the embedding $\Gamma(X,\Dcal)\hookrightarrow
\U^\wedge_\hbar/\I'_\hbar$ maps $\Gamma(X,\Dcal)_{fin}$ onto $(\U^\wedge_\hbar/\I'_\hbar)_{fin}^\Gamma$.

Pick an irreducible $G\times \K^\times$-module $L$. It is enough to show
a natural map
\begin{equation}\label{eq:map}\Hom_{G\times \K^\times}(L, \Gamma(X,\Dcal))\hookrightarrow \Hom_{G\times \K^\times}(L,(\U^\wedge_\hbar/\I'_\hbar)_{fin})\end{equation}
is an isomorphism onto $\Hom_{G\times \K^\times}(L,(\U^\wedge_\hbar/\I'_\hbar)_{fin}^\Gamma)$.

Consider the bundle $\Dsh_L:=\Dsh\otimes_{\K}L^*$ on $X$. This is a $\Dsh$-bimodule
(the direct sum of several copies of $\Dsh$).
 We have a connection $\widetilde{\nabla}_L:=\widetilde{\nabla}\otimes \operatorname{id}$.
 Of course, $\Dsh_L^{\widetilde{\nabla}_L}=\Dcal\otimes L^*$ and so $\Gamma(X,\Dsh_L)^{\widetilde{\nabla}_L}=\Gamma(X,\Dcal\otimes L^*)=\Hom(L,\Gamma(X,\Dcal))$.
 The group $G\times \K^\times$ acts naturally on $\Dsh_L$. This action gives rise
 to a map $\xi\mapsto \xi_{\Dsh_L}:\g\rightarrow \operatorname{End}(\Dsh_L)$. Then
 \begin{equation}\label{eq:connection_L}
\xi_{\Dsh_L}=\widetilde{\nabla}_{L,\xi_X}+ \frac{1}{\hbar^2}[\Phi_\hbar(\xi),\cdot]+\xi_{L^*}.
\end{equation}
Consider the restriction map $\Gamma(X,\Dsh_L)\rightarrow \Dsh_{L,x}$. Since $\Dsh_L$ is a pro-coherent
$G$-equivariant $\Str_X$-module, the restriction map descends to an isomorphism
\begin{equation}\label{eq:isomorphism_L}\Gamma(X,\Dsh_L)^G\xrightarrow{\sim} \Dsh_{L,x}^H= (\U^\wedge_\hbar/\I'_\hbar\otimes L^*)^H.\end{equation} Thanks to (\ref{eq:connection_L}),
on $\Gamma(X,\Dsh_L)^G$  the connection $\widetilde{\nabla}_L$ coincides with
$-\frac{1}{\hbar^2}[\Phi_\hbar(\xi),\cdot]-\xi_{L^*}$. Under the isomorphism (\ref{eq:connection_L})
the last operator corresponds to $-\frac{1}{\hbar^2}[\xi,\cdot]-\xi_{L^*}$. So, restricting
(\ref{eq:isomorphism_L}) to $\Gamma(X,\Dcal_L)^G$, we get an isomorphism
\begin{equation}\label{eq:iso}\Gamma(X,\Dcal_L)^G\xrightarrow{\sim} (\U_\hbar^\wedge/\I'_\hbar\otimes L^*)^\g \cap
(\U_\hbar^\wedge/\I'_\hbar\otimes L^*)^H,\end{equation} where $\g$ acts on
$\U_\hbar^\wedge/\I'_\hbar\otimes L^*$ by $\xi\mapsto \frac{1}{\hbar^2}[\xi,\cdot]+ \xi_{L^*}$.

Take $\K^\times$-invariants in (\ref{eq:iso}). The left hand side becomes $\Hom_{G\times\K^\times}(L,\Gamma(X,\Dcal))$.
By the definition of  the  $\Gamma$-action on $(\U^\wedge_\hbar/\I'_\hbar)_{fin}$,  the space $\K^\times$-invariants
in the right hand side of  is nothing else but $\Hom_{G\times \K^\times}\left(L,(\U^\wedge_\hbar/\I'_\hbar)_{fin}^\Gamma\right)$.
The corresponding map
$$\Hom_{G\times\K^\times}\left(L,\Gamma(X,\Dcal)\right)\rightarrow \Hom_{G\times \K^\times}\left(L,(\U^\wedge_\hbar/\I'_\hbar)_{fin}^\Gamma\right)\hookrightarrow \Hom_{G\times \K^\times}(L,\left(\U^\wedge_\hbar/\I'_\hbar\right)_{fin})$$ coincides with (\ref{eq:map}).
%Let $\widetilde{X}$ denote the universal covering of $\Orb$ (and of $X$), $\rho:\widehat{X}\twoheadrightarrow X$
%be the projection. Consider the quantization $\widetilde{\Dcal}$ of $\widetilde{X}$ corresponding to
%the ideal $\I'_\hbar$. This quantization is, in addition, $\Gamma$-equivariant (i.e, a natural
%$\Gamma$-action on $\Str_{\widetilde{X}}$ lifts to a $\Gamma$-action on $\widetilde{\Dcal}$ that commutes
%with $G\times\K^\times$ and fixes the quantum comoment map). It is clear from the construction in the previous
%subsection that $\Dcal=\rho_*(\widetilde{\Dcal})^\Gamma$ and hence %$\Gamma(X,\Dcal)=\Gamma(\widetilde{X},\widetilde{\Dcal})^\Gamma$. Also the image of the homomorphism
%$\Gamma(\widetilde{X},\widetilde{\Dcal})\rightarrow \U_\hbar^\wedge/\I'_\hbar$ consists
%of $\g$-locally finite elements. It follows from the construction of the embedding
%$\Gamma(\widetilde{X},\widetilde{\Dcal})\hookrightarrow (\U_\hbar^\wedge/\I'_\hbar)_{\g-fin}$ that it is
%$\Gamma$-equivariant.

%So we may replace $(X,\Dcal)$ with $(\widetilde{X},\widetilde{\Dcal})$ and assume that $\Gamma=\{1\}$.

%Now $(\U_\hbar^\wedge/\I'_\hbar\otimes L^*)^\g \subset
%(\U_\hbar^\wedge/\I'_\hbar\otimes L^*)^H$ because $H$ is connected. So we get
%an isomorphism $\Gamma(X,\Dcal_L)^G\xrightarrow{\sim} (\U^\wedge_\hbar/\I'_\hbar\otimes L^*)^{\g}$.
%But this isomorphism is nothing else but the embedding
%$\Hom_G(L,\Gamma(X,\Dcal))\rightarrow \Hom_\g(L, \U^\wedge_\hbar/\I'_\hbar)$
%we are interested in.
\end{proof}

\subsection{Comparison with Moeglin's results}\label{SUBSECTION_Moeglin}
An alternative language to speak about quantizations of coverings of nilpotent orbits is that of {\it Dixmier algebras}.
Recall that a Dixmier algebra $\A$ over $\U=U(\g)$ is an associative algebra equipped with a
(rational) $G$-action and a $G$-equivariant homomorphism $\U\rightarrow \A$ such that
$\A$ is a finitely generated left $\U$-module. Then automatically the differential
of the $G$-action coincides with the adjoint action of $\g$.

Let us define a notion of a filtered Dixmier algebra {\it quantizing a covering of $\Orb$}.
By definition, this is a pair $(\A,\F_\bullet \A)$, where $\A$ is a Dixmier algebra
and $\F_i\A, i\geqslant 0,$ is a $G$-stable increasing exhaustive algebra filtration on $\A$
satisfying the following conditions:
\begin{itemize}
\item The induced filtration on the image of $\U$ is compatible with the filtration
induced from the PBW filtration of $\U$.
\item The associated graded algebra $\gr\A$ is a finitely generated commutative domain.
\item There is a $G$-equivariant embedding $\gr\A\hookrightarrow \K[\widetilde{\Orb}]$, where
$\widetilde{\Orb}$ is the universal covering of $\Orb$, intertwining the natural homomorphisms
$S(\g)\rightarrow \gr \A, S(\g)\rightarrow \K[\widetilde{\Orb}]$ (the homomorphism
$S(\g)\rightarrow \gr\A$ is induced from a linear map $\g\rightarrow \gr\A$
that comes from the Lie algebra homomorphism $\g\rightarrow \A$).
\end{itemize}

One can introduce a partial order on the set of isomorphism classes of filtered Dixmier algebras:
$(\A,\F_\bullet\A)\preceq (\A',\F_\bullet\A')$ if there is a $G$-equivariant embedding $\iota:\A\hookrightarrow
\A'$ that is strictly compatible with the filtrations: $\iota^{-1}(\F_i\A')=\F_i\A$ for all $i$.

Moeglin, \cite{Moeglin2}, related maximal filtered Dixmier algebras quantizing a covering of $\Orb$ to
primitive ideals in $\U$ "admitting a Whittaker model".  Let us explain her result.

Recall the subgroup $M\subset G$ and the   $\U$-module $\U/\U\m_\chi$
  equipped with a Kazhdan filtration, see Subsection \ref{SUBSECTION_W_generalities}.
We remark that each $\U$-submodule in $\U/\U\m_\chi$ is automatically $M$-stable.
Let $\J$ be a primitive ideal in $\U$. Following Moeglin,
we say that an irreducible quotient $N$ of $\U/\U\m_\chi$ is a {\it Whittaker
model} for $\J$ if $\J$ annihilates $N$ and $\gr N$ is isomorphic to $\K[M]$
as a graded $M$-module (the grading on $\K[M]$ was introduced in Subsection \ref{SUBSECTION_W_generalities}).   Under the Skryabin equivalence
a quotient $N$ of $\U/\U\m_\chi$ with $\gr N=\K[M]$ corresponds to a one-dimensional
$\Walg$-module. This follows from Lemma \ref{Lem:GG_filtr}.

Let $N$ be a Whittaker model for $\J$. Let $L(N,N)$ denote the space of $\g$-finite maps
$N\rightarrow N$. This is an algebra equipped with a homomorphism $\U \rightarrow L(N,N)$.
Consider the filtration $\F_\bullet$ on $L(N,N)$ induced by the  filtration on $N$: $\KF_i L(N,N)$
consists of all maps $\varphi$ such that $\varphi(\KF_j N)\subset \KF_{i+j}N$
for all $j$. Then, according to Moeglin,
the pair $(L(N,N), \F_\bullet L(N,N))$ is a maximal filtered Dixmier algebra quantizing a covering of
$\Orb$ (see \cite{Moeglin2}, Theorem 15).

Let us briefly explain the relation between our construction and Moeglin's.

First of all, let $(\A,\F_\bullet \A)$ be a filtered Dixmier algebra. Form the Rees algebra
$R_\hbar(\A)$ and complete it with respect to the $\hbar$-adic topology. Then we can localize
this completion on $\Spec(\gr \A)$ to get a sheaf of algebras. The restriction of
this sheaf to the open $G$-orbit $\widetilde{\Orb}$ is a homogeneous Hamiltonian $G$-equivariant quantization $\Dcal$ of $\widetilde{\Orb}$. Now let $X\twoheadrightarrow \widetilde{\Orb}$ be a $G$-equivariant covering of
$\widetilde{\Orb}$. It is possible to show that there is a unique quantization of
$X$ that lifts $\Dcal$ (in an appropriate). Also it is not difficult to see that
if $(\A,\F_\bullet\A)\preceq (\A',\F'_\bullet\A)$, then the quantizations of
the open $G$-orbit in $\Spec(\gr \A')$ given by $\A$ and by $\A'$ are isomorphic.

On the other hand, let $\Dcal$ be a quantization of the universal cover
$X\twoheadrightarrow \Orb$. Consider the algebra $\Gamma(X,\Dcal)$ and its
finite part $\Gamma(X,\Dcal)_{fin}$. Set $\A_\Dcal:=\Gamma(X,\Dcal)_{fin}/(\hbar-1)\Gamma(X,\Dcal)_{fin}$.
By Proposition \ref{Prop:sections},  $\A_{\Dcal}=(\Walg/\I)^{\widetilde{\dagger}}$
in the notation of \cite{HC}, Remark 3.5.1. Using the techniques of \cite{HC}
it is easy to show that $\A_{\Dcal}$ (with its natural filtration induced by
the $\K^\times$-action on $\Gamma(X,\Dcal)_{fin}$) is a maximal filtered Dixmier algebra
quantizing a covering of $\Orb$. The construction of the previous paragraph shows
 that, conversely, any maximal filtered Dixmier algebra
has the form $\A_{\Dcal}$.

\end{document}